\documentclass[11pt]{amsart}
\usepackage[english]{babel}
\usepackage[utf8]{inputenc}
\usepackage{amsmath}
\usepackage{graphicx}
\usepackage{amssymb}
\usepackage{amsthm}
\usepackage{tikz-cd}
\usepackage{mathrsfs}
\usepackage[colorinlistoftodos]{todonotes}
\usepackage{enumitem}
\usepackage{yfonts}
\usepackage{xcolor}
\usepackage{mathtools}
\usepackage{hyperref}
\usepackage{comment}

\title{A Quantitative Selberg's Lemma}
\author{Tsachik Gelander, Raz Slutsky}

\date{}
\newtheorem{thm}{Theorem}[section]
\newtheorem{lem}[thm]{Lemma}
\newtheorem{prop}[thm]{Proposition}

\newtheorem{cor}[thm]{Corollary}

\newtheorem{rmk}[thm]{Remark}

\newcommand{\R}{\mathbb{R}}
\newcommand{\Q}{\mathbb{Q}}
\newcommand{\C}{\mathbb{C}}
\newcommand{\Z}{\mathbb{Z}}
\newcommand{\N}{\mathbb{N}}

\begin{document}

\begin{abstract}
We show that an arithmetic lattice $\Gamma$ in a semi-simple Lie group $G$ contains a torsion-free subgroup of index $\delta(v)$ where $v = \mu (G/\Gamma)$ is the co-volume of the lattice. We prove that $\delta$ is polynomial in general and poly-logarithmic under GRH. We then show that this poly-logarithmic bound is almost optimal, by constructing certain lattices with torsion elements of order $\sim \frac{\log v}{\log \log v}$.
\end{abstract}

\maketitle
\section{Introduction}
The classical Selberg's lemma states the following quite general result:

\begin{thm}
A finitely generated group $\Gamma \leq GL_n(k)$ where $k$ is a field of characteristic zero has a torsion-free subgroup of finite index.
\end{thm}

This lemma is very useful in algebra and geometry, where many times working with torsion-free groups, or equivalently, manifolds instead of orbifolds, is much easier. It is also related to Burnside's problem since it implies that linear finitely generated torsion groups are finite, a result originally due to Schur \cite{schur1911uber}.\\

In particular, this lemma is true for the case of $\Gamma$ a lattice in a semi-simple Lie group $H$. We recall that a lattice is a discrete subgroup such that the quotient $H/\Gamma$ carries a finite $H-$invariant measure. Lattices in Lie groups play a key role in the study of hyperbolic manifolds, locally symmetric spaces, number theory, and more. Our goal is to prove a quantitative result relating the co-volume of the lattice to the index of the torsion-free subgroup. We restrict to arithmetic lattices.  Recall that by the celebrated Margulis Arithmeticity Theorem all lattices in a semi-simple Lie group of higher-rank are arithmetic \cite{margulis1991discrete}.  Note that for non-uniform arithmetic lattices the required bound is a constant depending only on the ambient Lie group $G$ and not on the co-volume (see \cite[Lemma 13.1]{gelander2004homotopy}). However, for uniform lattices the bound on the index does grow with the co-volume as shown in \S 3 below.
Assuming the generalized Riemann hypothesis (GRH) we prove:

\begin{thm}[GRH]\label{thm: intro GRH main theorem}
    Let $\Gamma$ be an arithmetic lattice in a semi-simple Lie group $H$ of co-volume $v$ and let $\varepsilon >0$. Then there exists a constant $c = c(H, \varepsilon)$ such that $\Gamma$ contains a normal subgroup $\Delta$ which is torsion-free and such that $[\Gamma : \Delta] \leq c (\log v)^{(2+\varepsilon)\dim H}$.
\end{thm}

Removing the GRH assumption, we have:

\begin{thm}\label{thm: intro unconditional main theorem}
Let $\Gamma$ be an arithmetic lattice in a semi-simple Lie group $H$ of co-volume $v$. Then there exists a constant $c = c(H)$ such that $\Gamma$ contains a normal subgroup $\Delta$ which is torsion-free and such that $[\Gamma : \Delta] \leq v^{c \dim H}$.
    
\end{thm}

This is in line with the many works relating the co-volume of lattices in semi-simple Lie groups to their complexity, be it algebraic, geometric or topological. See for example the work of the first author relating the co-volume to the topological complexity of the quotient in \cite{gelander2004homotopy}, works relating the co-volume to the minimal number of generators such as \cite{VVR,gelander2020minimal, lubotzky2022asymptotic}, and many more, for example, \cite{BGS85, CHM, CAM, 7s, abert2017rank, BL12, BL19, BGS20, abert2022homology, abert2023Betti, fraczyk2022growth,  frkaczyk2022homotopy}, and versions for general hyperbolic groups in \cite{lazarovich2021volume, lazarovich2023finite}.

Many of the above results work only in the case where $\Gamma$ is torsion-free since they use geometric methods to analyze the structure of the associated locally-symmetric spaces. Some efforts to extend these results to the case of general lattices are carried out in \cite{samet2013betti,emery2014torsion} for example. Usually, this requires a deep understanding of the geometry around the singular points of the orbifold. Hopefully, our quantitative version of Selberg's lemma will be useful in extending such torsion-free results to general lattices and locally symmetric orbifolds more easily, by passing to a torsion-free subgroup of controlled index in terms of the co-volume. In \S \ref{section:application} below, we shall elaborate on one example of possible application.\\

We remark that in our theorems we prove that there exists a \emph{normal} subgroup which is torsion-free and of small index, and so a natural question is whether there exists a subgroup which is not necessarily normal but of even smaller index.\\

In the second part of this paper, we show a lower bound for the minimal index of a torsion-free subgroup by constructing lattices with torsion elements.  This lower bound shows that the conditional upper bound is almost optimal, namely:

\begin{thm} \label{thm: intro lower bound}
Let $G = SO(p,q)$,  where $p+q \geq 3$.  Then there exist a constant $c>0$ and a sequence of lattices $\Gamma_n$ of volumes $v_n$ such that any torsion-free subgroup  in $\Gamma_n$ is of index at least $c \frac{\log v_n }{\log \log v_n}$.
\end{thm}

This is proved in \S \ref{section:lower bound}. 

\subsection*{Non-arithmetic Lattices} Our result raises the question whether a similar bound can be obtained for non-arithmetic lattices. The answer to that is no.  For example, in $\mathrm{SL}_2(\R)$, one can construct Fuchsian groups of signature $(g, m_1,...,m_k)$, where $g$ is the genus and $m_1,...,m_k$ are the orders of torsion elements.  The Hurwitz-Riemann formula for the co-volume of such lattices is given by $2g-2 + \sum_{i = 1}^k (1 - \frac{1}{m_i})$. In particular, one can take a sequence of triangle groups, i.e., lattices of signature $(0,m_1,m_2,m_3)$ and letting $m_i$ go to infinity, we get lattices with bounded volume but growing torsion.  It is been long known that only finitely-many of those can be arithmetic,  see \cite{takeuchi1977arithmetic} as well as \cite{borel1981commensurability}. An even more surprising construction is given in  \cite{jones1998minimal},  where the authors construct a sequence of lattices in $\mathrm{PSL}_2(\C)$ with the property that the size of torsion subgroups is bounded, the volume is bounded, but the minimal index of a torsion-free subgroup grows to infinity.  In this construction as well, only finitely many of these lattices can be arithmetic. 

\section*{Acknowledgements}
The second author would like to thank Oren Becker for discussions regarding \S \ref{sec:upper bound}. The authors wish to thank Nir Avni for very useful suggestions regarding \S \ref{section:lower bound}.

 \section{Upper Bound}\label{sec:upper bound}
 
 The setting of this paper is that of arithmetic lattices. We recall the definition. Let $H$ be a simple connected linear Lie group,
 and let $G$ be a simple, simply connected, connected algebraic group defined over a number field $k$,
 with an epimorphism 
 $\phi: G( k \otimes_\Q \R) \rightarrow H$ whose kernel is compact.
  Then $\phi(G(\mathcal{O}))$ and subgroups of $H$ which are commensurable to it are called \emph{arithmetic}. Such $G$ will be called admissible. By \cite{margulis1991discrete}, all lattices in higher-rank arise in this way, and by \cite{Cor92,GS92} the same is true for some of the rank one groups.\\
 
 We are going to be interested in the structure of $\mathcal{O}$, the ring of integers of a number field $k$ of degree $d$. Our strategy is to find a congruence subgroup of small index, since congruence subgroups are generally torsion-free. For example, when $\mathcal{O} = \Z$, a principal congruence subgroup of level $p$ is always torsion-free for $p>2$, see \cite[III.2.3]{kionke2012thesis}. 
 Since the index is equal to the size of the finite quotient, we would like to find an ideal $I$ in $\mathcal{O}$ such that $\mathcal{O} / I$ is small, and hence also $G(\mathcal{O}/I)$ is small.
One way to do this is to look for a rational prime $p$ which splits completely in $\mathcal{O}$, because then
  $(p) = \mathfrak{p}_1 \cdot \cdot \cdot \mathfrak{p}_d$ with $\mathcal{O}/\mathfrak{p}_1 \cong \mathbb{Z}/p \mathbb{Z}$. In other words, taking $I = \mathfrak{p_1}$, the quotient has size $p$, even if the degree is very large.

For a number field $k$ we denote by $d$ its degree, and $D_k$ the absolute value of its discriminant. We first prove the following lemma:
\begin{lem}[GRH] \label{lem: Main lemma for upper bound}
    Fix $\varepsilon > 0 $ and let $G$ be a simple, simply connected, connected algebraic group defined over $k$ and let $\mathcal{O}$ be its ring of integers. Then there exists $\Gamma_1 \lhd G(\mathcal{O})$  which is torsion-free and such that $[ G(\mathcal{O}) : \Gamma_1] \leq C(d + \log D_k)^{(2+\varepsilon)\dim G}$ where $C$ depends only on $\varepsilon$.
\end{lem}

\begin{proof}
We are going to choose a principal congruence subgroup with two properties. The first is that it is torsion-free, and the second is that its index (or equivalently, the size of the quotient) is small relative to the arithmetic data of $\mathcal{O}$. 
By \cite[III.2.3]{kionke2012thesis}, the principal congruence subgroup defined by an ideal $\mathfrak{a}$ is torsion-free if $\mathfrak{a}$ has the property that $\mathfrak{a}^{p-1}$ does not divide $p\mathcal{O}$ for every rational prime $p$.
Since $p\mathcal{O}$ decomposes as prime ideals with ramification index at most $d$, 
if $\mathfrak{a}$ sits above $p > d$,
this ensures that $\mathfrak{a}^{p-1}$ will not divide $p\mathcal{O}$, and hence the principal congruence subgroup defined by $\mathfrak{a}$ is torsion-free. Second, we wish to choose $\mathfrak{a}$ with a small inertia degree, that is, with $[\mathcal{O} / \mathfrak{a} : \Z / p \Z]$ being small.  This will be done using effective estimates on the number of prime ideals of small norms. 

By Landau's Prime Ideal Theorem, the number of prime ideals in $k$ of norm at most $x$ is at least $\mathrm{Li}(x) - \mathrm{Err}(x)$ for some error term. Among those, only $d^2$ can sit above a prime which is smaller than $d$, thus, we get our required prime ideal as soon as $\mathrm{Li}(x) > \mathrm{Err}(x) + d^2 $. By \cite[Cor. 1.4]{grenie2019explicit}, (which depends on GRH), one can take
\[
\mathrm{Err}(x) \leq 13\sqrt{x}(\log(D_k) + d\log(x))
\]

Plugging this estimate on the error with the above inequality, together with the fact that 
$\mathrm{Li}(x) > \frac{x}{\log x}$ for large enough $x$, we get that such a prime ideal exists as soon as $x > C_1(\log D_k + d)^{2+\varepsilon}+d^2 > C_2(\log D_k+d)^{2+\varepsilon}$, for $C_2$ depending on $\epsilon$.

By the choice of $\mathfrak{a}$, we know that $G(\mathfrak{a})$ is torsion-free. We have that $[G(\mathcal{O}: G(\mathfrak{a})] \leq |\mathcal{O} / \mathfrak{a}|^{\dim G}$. But the norm of $\mathfrak{a}$ is exactly the cardinality of the field $\mathcal{O} / \mathfrak{a}$, which is $C_2(\log D_k+d)^{2+\varepsilon}$. And the lemma is proved.
\end{proof}

We are now ready to prove Theorem \ref{thm: intro GRH main theorem}. 
\begin{proof}[Proof of Thm. \ref{thm: intro GRH main theorem}]
Let $\Gamma$ be an arithmetic lattice in $H$.
Then there exists $G$, an admissible algebraic group defined over $k$, such that $G( k \otimes_\Q \R)$ is isogenous to $H \times K$ for $K$ a compact Lie group, and a lattice $\tilde{\Gamma} \subset G(k)$ (\cite[Prop. 1.2]{ryzhikov1997classification}) which is commensurable to $G(\mathcal{O})$, such that the projection of $\tilde{\Gamma}$ to $H$ is equal to $\Gamma$.
Moreover, one can change the integral basis such that $\tilde{\Gamma} \subset G(\mathcal{O})$, see \cite[Lemma 6]{vinberg1971rings}.
Since the index of the image of a finite-index subgroup of $\tilde{\Gamma}$ will only decrease, it is enough to find a torsion-free subgroup of small index in $\tilde{\Gamma}$, and in fact, in $G(\mathcal{O})$. Moreover, the image of a torsion-free discrete group by a map with a compact kernel $K$ is again torsion-free, since if $\varphi(\gamma^n) = e$ then $\gamma^n \in K$ and is of finite order as well. By Lemma \ref{lem: Main lemma for upper bound}, there exists a torsion-free subgroup of index at most $C(d+\log D_k)^{(2+\varepsilon)\dim G}$ in $G(\mathcal{O})$. By Prasad's volume formula \cite{prasad1989volumes}, and \cite[Sec. 3.3]{belolipetsky2007counting}, there exist constants $c_1, c_2$ which depend only on $H$, such that
$\log D_k \leq c_1 \log \mathrm{vol}(H/\Gamma)$ and  $d \leq c_2 \log \mathrm{vol}(H/ \Gamma)$, hence the index of our torsion-free subgroup is at most $c_3(\log \mathrm{vol}(H/\Gamma))^{(2+\varepsilon)\dim H} $
\end{proof}
~\\

The case for Theorem \ref{thm: intro unconditional main theorem} is much easier. While it is possible to give some unconditional bounds on the norms of prime ideals, for our purposes such bounds do not yield better results than just looking at the ideal $p\mathcal{O}$ for $p > d$. One can also forget about the field $k$ and look at the $\Z$ points of the restriction of scalars of $G$. That is, in the notation of the proof above, $\tilde{\Gamma} \leq G_1(\Z)$ where $G_1$ is defined over $\Q$, and $\dim G_1 = d \cdot \dim H$. The $3-$rd congruence subgroup is already torsion-free, and its index is $3^{d \dim H}$. By the same argument as before, this translates into $[\Gamma : \Gamma_1 ] \leq v^{c \dim H}$.
\qed

\section{Lower Bound}\label{section:lower bound}
In this section, we will build a sequence of lattices with a relatively large torsion subgroup in terms of their co-volume. This will show a lower bound for the index of a torsion-free normal subgroup since the index  must be larger than the size of every finite subgroup. Our lower bound is of order $\sim \frac{\log v}{\log \log v}$.
We are going to construct elements of prime order for simplicity, since the estimates on the discriminant are nicer, but the following lemma can be stated in exactly the same way for every natural number.

\begin{lem}
    Let $G = SL_2(\R)$. Then there exists a sequence of arithmetic lattices $\Gamma_p$, where $p$ runs over all prime numbers, which contain an element of order $p$, and such that $\operatorname{vol}(G/\Gamma_p) = O(p^p)$.
\end{lem}

\begin{proof}
    We are going to build an arithmetic lattice in $SL_2(\R)$ using the quadratic form $f(x,y,z) = x^2 + y^2 - cz^2$ where $c$ will be some positive real number we shall soon choose. The group $SO(f)$ will be isogenous to $SL_2(\R)$. Choosing $c$ to be in some totally real number field $K$ with a ring of $S$-integers $\mathcal{O}(S)$ will define the discrete group $G(\mathcal{O}(S))$ which is a lattice in the group $G_S = \prod_{v \in S} G_{K_v}$.
    If we make sure that at all valuations, except the standard archimedean one, the groups $G_{K_v}$ are anisotropic, we will get that the projection onto $SO(f)$, or $SL_2(\R)$, is a lattice.
    Specifically, we will choose $S$ to be the set of archimedean valuations together with the primes above $2$. The number field will be $\Q[\omega]$ where $\omega = \cos{\frac{2\pi}{p}}$.
    The minimal polynomial of $\cos{\frac{2\pi}{p}}$ is related to the Chebyshev Polynomial $T_p$, see \cite{liang1976integral, watkins1993minimal}. It is defined over $\Z$ and has a leading coefficient which is a power of $2$, hence its zeroes are $\mathcal{O}(S)$-integers for our choice of $S$. This implies that the rotation matrix of order $p$ is in $G(\mathcal{O}(S))$.

    We choose $c = T + \omega$ where $T \in \Q$ and satisfies the following two properties: First, choose $T$ such that $ \cos(\frac{3\pi}{p}) < -T < \cos(\frac{2\pi}{p})$. Second, choose $T = a/b$ such that the $2$-adic valuation of $T+ \omega$ is odd. Since the first condition is an open interval, finding a $T$ which also satisfies the second condition is possible.
    
    This will imply that there are no non-trivial solutions for $x^2 + y^2 - cz^2 = 0$ over all $K_v$ except for the standard Euclidean valuation. For archimedean valuations, this is clear since the Galois conjugates of $\omega$ are $\cos(\frac{k\pi}{p})$, and thus $T+\sigma(\omega)$ is positive for any non-trivial field automorphism $\sigma$.
    For the $2$-adic valuation, there are no solutions since the $2$-valuation of $c$ will be odd, and thus the equation $x^2+y^2 = cz^2$ can not have non-trivial solutions. We thus get that the projection of $G(\mathcal{O}(S))$ onto $SL_2(\R)$ is still a lattice, and we denote it by $\Gamma_p$. We have seen that it contains an element of order $p$.\\

    Finally, we compute the co-volume of $\Gamma_p$. By a result of Mahler, see \cite{liang1976integral}, the discriminant of $K$ is given by $p^{\frac{p-2}{2}}$. Moreover, by \cite{belolipetsky2007counting}, there exist constants $a,b$, such that $\operatorname{vol}(G/ \Gamma_p) \leq a D_K^{b}$, hence 
    $$ 
    \log \operatorname{vol}(G/\Gamma_p) \leq c \: p \log p.
    $$

\end{proof}

\begin{rmk}
    We note that the same construction can be done in higher rank, i.e., in $SO(m,n)$ for $m,  n \geq 2$. 
\end{rmk}

We thus get an estimate on the lower bound, at least for $SO(m,n)$, where $m+n \geq 3$. Since the size of a finite subgroup is a lower bound on the index of a torsion-free subgroup, we get Theorem \ref{thm: intro lower bound}:
\begin{cor}
  A bound on the index of a torsion-free normal subgroup in terms of the co-volume of a lattice in $SO(n,m)$ can not be smaller than  $\sim \frac{\log v}{\log \log v}$.
\end{cor}

\section{Size of Finite subgroups}

For completeness, we give the following bound on the size of finite subgroups of an arithmetic lattice, which is probably known to experts. Since the index of a torsion-free subgroup bounds this quantity, we wanted to give an unconditional bound for this quantity as well.

\begin{lem}\label{torsion_in_glnk}
Let $k$ be a number field of degree $d$, and let $g$ be an element of order $\ell < \infty$ in $GL_n(k)$, then $\ell \leq 2(nd)^{2n}$.

\end{lem}

\begin{proof}
Suppose $g^\ell = 1$. This implies that the minimal polynomial of $g$, $p_g \in k[x]$, divides $x^\ell-1$, which has distinct roots of unity. This means that the eigenvalues of $g$ are roots of unity of orders $m_1,...,m_s$ with $\operatorname{lcm}(m_1,...,m_s) = \ell$.
We write $p_g = f_1 \cdot ... \cdot f_t$ where $f_i$ are irreducible over $k$. Since $\deg p_g \leq n$, we have that 
\begin{align}\label{degrees_inequality}
    \sum_{i = 1}^t \deg f_i \leq n ,
\end{align}
but $\deg f_i = [k[\zeta_i] : k]$ for some $\zeta_i$ being a primitive root of unity of order $m_i$. \\
Denote by $\varphi$ the Euler totient function, and note that the diagram
\[\begin{tikzcd}
	& k \\
	{\mathbb{Q}} && {k[\zeta_i]} \\
	& {\mathbb{Q}[\zeta_i]}
	\arrow["d", hook, from=2-1, to=1-2]
	\arrow["{\varphi(m_i)}"', hook, from=2-1, to=3-2]
	\arrow["{\leq d}"', hook, from=3-2, to=2-3]
	\arrow[hook, from=1-2, to=2-3]
\end{tikzcd}\]
implies that $\varphi(m_i) \leq d \cdot [k[\zeta_i] : k] \leq d \cdot \varphi(m_i)$, and so, multiplying (\ref{degrees_inequality}) by $d$ we get

$$\sum_{i=1}^t \varphi(m_i) \leq nd  $$
and in particular $\varphi(m_i) \leq nd$ for every $i$. Since $\operatorname{lcm}(m_1,...,m_s) = \ell$ we have $$\varphi(\ell) \leq \varphi(m_1) \cdot ... \cdot \varphi(m_s) \leq (nd)^n $$
Finally, $\varphi(l) \geq \frac {\sqrt{l}} {\sqrt 2} $ and so we get that $\ell \leq 4 (nd)^{2n}$
\end{proof}

\begin{lem} \label{lem: size of torsion elements}
Let $\Gamma \leq H$ be an arithmetic lattice of co-volume $v$ and let $a  \in \Gamma$ be an element of order $\ell < \infty$, then 
there exist constants $c_1, c_2>0$ depending only on $H$ such that $ \ell \leq  c_1 (\log v)^{c_2} $

\end{lem}

\begin{proof}
We may devide by the finite center and assume $H$ is centre-free, and thus $\Gamma \leq GL_s(k)$ for some $s$ depending only on $H$. By Lemma $\ref{torsion_in_glnk}$, $\ell \leq 2(sd)^{2s}$. By Prasad's volume formula and \cite[Sec. 3.3]{belolipetsky2007counting}, there exists a constant $C = C(H)>0$ such that $d \leq C \log v$, and so, 
$$ \ell \leq c_1 (\log v)^{c_2} $$
\end{proof}

\begin{prop} \label{prop: size of finite subgroups}
Let $\Gamma \leq H$ be an arithmetic lattice as above, then there exist $a, b$, depending only on $H$ such that every finite subgroup $F \leq \Gamma$ is of size at most $a (\log v)^{b}$

\end{prop}
\begin{proof}
    By Jordan's Theorem \cite{jordan1878memoire}, every finite subgroup of $GL_n(\C)$ contains an abelian subgroup of index at most $i$, where $i$ depends only on $n$. Thus, it is enough to bound the size of a finite abelian subgroup of $F$. Such a subgroup is diagonalizable, and every element is of order at most $c_1(\log v)^{c_2}$ by Lemma \ref{lem: size of torsion elements}, hence the size of $F$ is at most $(c_1(\log v)^{c_2})^n \leq a (\log v^{b})$ for some $a,b$ depending only on $H$.
\end{proof}

\section{Possible application}\label{section:application}
Let us give one example of a possible result using the recent breakthrough of \cite{fraczyk2023poisson}. In this paper, the authors show that the minimal number of generators of higher-rank torsion-free lattices is sub-linear in the co-volume, settling a conjecture of Abert, Gelander and Nikolov  \cite{abert2017rank} for this class of lattices. However, the case of higher-rank lattices with torsion is still open.
We thus note that almost any quantitative bound on the number of generators of a torsion-free lattice, combined with our result, would settle this case as well. This is because if $\Gamma_1 \leq \Gamma$ is a torsion-free normal subgroup of index $m$,
$d(\Gamma) \leq d(\Gamma_1) + d(\Gamma/ \Gamma_1)$ and $\mathrm{vol}(\Gamma_1) = m \cdot \mathrm{vol}(\Gamma)$. Denote the bound on the number of generators of a torsion-free lattice in terms of the co-volume by $f(v)$. Using Theorem \ref{thm: intro GRH main theorem} and the fact that a finite group of order $n$ can be generated by $\log(n)$ elements, we get:

\[ \frac{d(\Gamma)}{v} \leq \frac{d(\Gamma_1) + \log(m)}{v}  \leq \frac{f(v \log^{c}(v)) + \log \log v}{v}\]

Thus, any quantitative bound with the property that 
$f(v \log^{c}(v))$ is sub-linear in $v$ will do (conditionally). In particular, a bound of the form $d(\Gamma) \leq v^{1-\alpha}$ for some positive $\alpha$, which is considered quite possible. The actual bound for the number of generators is expected to be logarithmic in the volume, see \cite{lubotzky2022asymptotic}, where this is established for non-uniform lattices. A similar computation using Theorem \ref{thm: intro unconditional main theorem} will give an unconditional bound as long as $d(\Gamma_1) \leq v^{1/(c \dim H+1)}$. For example, any poly-logarithmic bound will do.
  

\bibliography{mybibliography}
\bibliographystyle{alpha}

\end{document}